\documentclass[12pt]{amsart}
\usepackage{amssymb,eucal,enumerate}
\usepackage{mathrsfs}
\usepackage{hyperref}
\usepackage[utf8]{inputenc}

\makeatletter
\def\newrefformat#1#2{%
  \@namedef{pr@#1}##1{#2}}
\def\fref#1{\@prettyref#1:}
\def\@prettyref#1:#2:{%
  \expandafter\ifx\csname pr@#1\endcsname\relax%
    \PackageWarning{prettyref}{Reference format #1\space undefined}%
    \ref{#1:#2}%
  \else%
    \csname pr@#1\endcsname{#1:#2}%
  \fi%
}
\makeatother

\newrefformat{eq}{\textup{(\ref{#1})}}
\newrefformat{ax}{\textup{\ref{#1}}}
\newrefformat{chap}{Chapter~\ref{#1}}
\newrefformat{sec}{Section~\ref{#1}}
\newrefformat{apdx}{Appendix~\ref{#1}}
\newrefformat{tab}{Table~\ref{#1} on page~\pageref{#1}}
\newrefformat{fig}{Figure~\ref{#1} on page~\pageref{#1}}

\newcommand{\mynewthm}[3][dummythm]{%
  \newtheorem{#2}[#1]{#3}%
  \newrefformat{#2}{#3~\ref{##1}}%
}

\theoremstyle{plain}
\mynewthm{thm}{Theorem}
\mynewthm{prp}{Proposition}
\mynewthm{lem}{Lemma}
\mynewthm{fct}{Fact}
\mynewthm{cor}{Corollary}

\theoremstyle{definition}
\mynewthm{dfn}{Definition}
\mynewthm{conv}{Convention}
\mynewthm{ntn}{Notation}
\mynewthm{conj}{Conjecture}
\mynewthm{cst}{Construction}

\theoremstyle{remark}
\mynewthm{rmk}{Remark}
\mynewthm{qst}{Question}
\mynewthm{exm}{Example}

\newcounter{cycprfcnt}
{\begin{list}{\PackageWarning{pezz}{Label required for cycprf}}%
  {%
    \setcounter{cycprfcnt}{1}
    \setlength{\itemindent}{0.5\leftmargin}%
    \setlength{\leftmargin}{0pt}%
  }%
}%
{\qedhere\end{list}}%

\makeatletter

\def\indsym#1#2{%
  \setbox0=\hbox{$\m@th#1x$}%
  \kern\wd0%
  \hbox to 0pt{\hss$\m@th#1\mid$\hbox to 0pt{$\m@th#1^{#2}$}\hss}%
  \lower.9\ht0\hbox to 0pt{\hss$\m@th#1\smile$\hss}%
  \kern\wd0}

\def\nindsym#1#2{%
  \setbox0=\hbox{$\m@th#1x$}%
  \kern\wd0%
  \hbox to 0pt{\hss$\m@th#1\not$\kern1.4\wd0\hss}
  \hbox to 0pt{\hss$\m@th#1\mid$\hbox to 0pt{$\m@th#1^{#2}$}\hss}%
  \lower.9\ht0\hbox to 0pt{\hss$\m@th#1\smile$\hss}%
  \kern\wd0}

\def\dotminussym#1#2{%
  \setbox0=\hbox{$\m@th#1-$}%
  \kern.5\wd0%
  \hbox to 0pt{\hss\hbox{$\m@th#1-$}\hss}%
  \raise.6\ht0\hbox to 0pt{\hss$\m@th#1.$\hss}%
  \kern.5\wd0}
\newcommand{\dotminus}{\mathbin{\mathpalette\dotminussym{}}}

\makeatother

\renewcommand{\emptyset}{\varnothing}

\DeclareMathOperator{\id}{id}

\newcommand{\cL}{\mathcal{L}}

\newcommand{\cM}{\mathcal{M}}
\newcommand{\cP}{\mathcal{P}}

\newcommand{\cT}{\mathcal{T}}

\newcommand{\bP}{\mathbb{P}}

\newcommand{\setR}{\mathbb{R}}

\newcommand{\setQ}{\mathbb{Q}}

\newcommand{\half}[1][1]{\frac{#1}{2}}

\theoremstyle{plain}
\newtheorem{theorem}{Theorem}[section]
\newtheorem{proposition}[theorem]{Proposition}
\newtheorem{lemma}[theorem]{Lemma}
\newtheorem{corollary}[theorem]{Corollary}
\theoremstyle{definition} 
\newtheorem{definition}[theorem]{Definition}

\newtheorem{notation}[theorem]{Notation}
\newtheorem{remark}[theorem]{Remark}

\renewcommand{\phi}{\varphi}
\newcommand{\sci}{\ensuremath{\sup\bigwedge\inf}}
\newcommand{\sub}{\operatorname{sub}}
\renewcommand{\varepsilon}{\epsilon}

\newcommand{\PMS}{\cP(\cM,\Sigma)}

\title{Model theoretic forcing in analysis}

\author{Ita\"\i{} \textsc{Ben Yaacov}}

\address{Ita\"\i{} \textsc{Ben Yaacov} \\
  Universit\'e Claude Bernard -- Lyon 1 \\
  Institut Camille Jordan \\
  43 boulevard du 11 novembre 1918 \\
  69622 Villeurbanne Cedex \\
  France}

\urladdr{http\string://math.univ-lyon1.fr/\textasciitilde begnac/}

\author{Jos\'e Iovino}

\address{Jos\'e Iovino\\
	Department of Mathematics\\
  The University of Texas at San Antonio\\
  One UTSA Circle
  San Antonio, TX 78249\\
  USA}

\email{iovino@math.utsa.edu}

\thanks{First author partially supported by NSF grant DMS-0500172.}

\date{\today}

\begin{document}

\begin{abstract}
We present a framework for model theoretic forcing in a non-first-order context, and present some applications of this framework to Banach space theory.
\end{abstract}

\maketitle

\section*{Introduction}

In this paper we introduce a framework of model theoretic forcing for
metric structures, i.e., structures based on metric spaces.
We use the language of infinitary continuous logic, which we define
below.
This is a variant of finitary continuous logic which is exposed in
\cite{BenYaacov-Usvyatsov:localstability} or
\cite{BenYaacov-Berenstein-Henson-Usvyatsov:metricstructures}.

The model theoretic forcing framework introduced here is analogous to that developed by Keisler~\cite{Keisler:1973} for structures of the form considered in first-order model theory.

The paper concludes with an application to separable quotients of Banach spaces.
The long standing Separable Quotient Problem is whether for every nonseparable Banach space $X$ there exists a operator $T\colon X\to Y$ such that $T(X)$ is a separable, infinite dimensional Banach space. We prove the following result (Theorem~\ref{T:separable quotient}): If $X$ is an infinite dimensional Banach space and $T\colon X \to Y$ is a surjective operator with infinite dimensional kernel, then there exist Banach spaces $\Hat{X},\Hat{Y}$ and a surjective operator $\Hat{T}\colon\Hat{X}\to\Hat{Y}$ such that
\begin{enumerate}
\item
$\Hat{X}$ has density character $\omega_1$,
\item
The range of $\Hat{T}$ is separable,
\item
$(X,Y,T)$ and $(\Hat{X},\Hat{Y},\Hat{T})$ are elementarily equivalent as metric structures.
\end{enumerate}

The paper is organized as follows. In \fref{sec:preliminaries}
we introduce the syntax that will be used in the paper. In \fref{sec:forcing}, we introduce model theoretic forcing for metric structures. 
\fref{sec:special forcing properties} we focus our attention on two particular forcing properties. These properties are used in \fref{sec:omitting types} to prove the general Omitting Types Theorem. The last section, \fref{sec:separable quotients}, is devoted to the aforementioned application to separable quotients.

For the exposition of the material we focus on one-sorted languages. However, as the reader will notice, the results presented here hold true, mutatis mutandi, for multi-sorted contexts. In fact, the structures used in the last section are multi-sorted. 

The authors are grateful to Yi Zhang for his encouragement and patience.

\section{Preliminaries}
\label{sec:preliminaries}

Recall that if
$f\colon (X,d) \to (X',d')$ is a mapping between two metric
spaces, then $f$ is uniformly continuous if and only if 
there exists a mapping 
$\delta\colon (0,\infty)\to (0,\infty]$ such that for all $x,y \in X$ and
$\varepsilon > 0$,
\begin{gather}
  \label{eq:UnifCont}
  d(x,y) < \delta(\varepsilon) \Longrightarrow d'(f(x),f(y)) \leq \varepsilon.
\end{gather}
If \eqref{eq:UnifCont} holds, we say that $\delta$ is a \emph{uniform continuity modulus} and that $f$ \emph{respects} $\delta$. The choice of strict and weak inequalities here is so that the
property of respecting $\delta$ be preserved under certain important constructions (e.g., completions and 
ultraproducts).

Let $\delta'\colon (0,\infty) \to (0,\infty]$ be any mapping, and define:
\begin{gather}
  \label{eq:GoodContMod}
  \delta(\varepsilon) = \sup \{\delta'(\varepsilon')\mid 0<\varepsilon'<\varepsilon\}.
\end{gather}
Then $\delta$ and $\delta'$ are equivalent as uniform continuity moduli,
in the sense that a function $f$
respects $\delta$ if and only if it respects $\delta'$.
In addition we have
\begin{gather}
  \label{eq:GoodContMod2}
  \delta(\varepsilon) = \sup \{\delta(\varepsilon')\mid 0<\varepsilon'<\varepsilon\},
\end{gather}
i.e.,
$\delta$ is increasing and continuous on the left.
As a consequence, \fref{eq:UnifCont} is equivalent to the apparently
stronger version:
\begin{gather}
  \label{eq:UnifContStrict}
  d(x,y) < \delta(\varepsilon) \Longrightarrow d'(f(x),f(y)) < \varepsilon.
\end{gather}
From this point on, when referring to a uniform continuity modulus $\delta$, we
mean one that satisfies \fref{eq:GoodContMod2}.

In this section we introduce infinitary continuous formulas.
For a general text regarding continuous structures and
finitary continuous first order formulas
we refer the reader to
Sections 2 and 3 of~\cite{BenYaacov-Usvyatsov:localstability} or
Sections~2--6
of~\cite{BenYaacov-Berenstein-Henson-Usvyatsov:metricstructures}.

Recall that a continuous signature $\cL$  consists of the following
data:
\begin{itemize}
\item For each $n$, a set of $n$-ary function and predicate symbols.
\item A distinguished binary predicate symbol $d$.
\item For each $n$-ary symbol $s$ and $i<n$, a
  uniform continuity modulus for the $i$th argument denoted $\delta_{s,i}$.
\end{itemize}

A continuous $\cL$-structure is a set $M$ equipped with
interpretations of the symbols of the language:
\begin{itemize}
\item Each $n$-ary function symbol is interpreted by an
  $n$-ary function:
  $$f^M\colon M^n \to M.$$
\item Each $n$-ary predicate symbol is interpreted by a continuous
  $n$-ary predicate:
  $$P^M\colon M^n \to [0,1].$$
\item The interpretation $d^M$ of the distinguished symbol $d$ is a
  complete metric.
\item For each $n$-ary symbol $s$ and $i < n$, the interpretation
  $s^M$, viewed as a function of its $i$th argument,
  respects the uniform continuity modulus $\delta_{s,i}$.
\end{itemize}

It is proved in \cite{BenYaacov-Usvyatsov:localstability} that the
following system of connectives is \emph{full}:
\[
x\mapsto \lnot x, \qquad x\mapsto \frac{x}{2}, \qquad (x,y)\mapsto x\dotminus y := \max(x-y,0)
\]
This means that for every $n \geq 1$,
the family of functions from $[0,1]^n \to [0,1]$ which can be
written using these three operations is dense in the class of all
continuous functions $[0,1]^n \to [0,1]$.
For the purposes of this paper (namely, to simplify the treatment of
forcing, in \fref{sec:forcing}), it is convenient to use the
connective $\dotplus$ instead of $\dotminus$.
Note that this  causes no loss in expressive power, since
$x\dotminus y =\lnot(\lnot x\dotplus y)$.

In this paper we extend the class of first-order continuous formulas by considering formulas that may contain the infinitary connectives $\bigwedge$ and $\bigvee$, where for a set of formulas $\Phi$, $\bigwedge_{\phi\in\Phi} \phi$ and $\bigvee_{\phi\in\Phi} \phi$ stand for $\sup\{\phi\mid\phi\in\Phi\} \phi$ and $\inf\{\phi\mid\phi\in\Phi\}$, respectively. Because of the infinitary nature of this language, in order to form formulas with these connectives, one needs to be
particularly careful about the uniform continuity moduli of the terms
and formulas with respect to each variable, denoted
$\delta_{\tau,x}$ and $\delta_{\varphi,x}$, respectively; thus, we have the following definition.

\begin{definition}
Let $\cL$ be a continuous signature. We define the formulas of
$\cL_{\omega_1,\omega}$. Simultaneously, for each variable $x$, each term $\tau$ and each formula $\phi$ of $\cL_{\omega_1,\omega}$ we define uniform continuity moduli $\delta_{\tau,x}$ and $\delta_{\varphi,x}$. Both definitions are inductive.
\begin{itemize}
\item A variable is a term, with $\delta_{x,x} = \id$ and
  $\delta_{x,y} = \infty$ for $y\neq x$.
\item If $f$ is an $n$-ary function symbol and $\tau_0,\dots,\tau_{n-1}$ are terms, then $f\tau_0\ldots\tau_{n-1}$ is a term. If $\tau$ is a term of this form,
  \begin{gather*}
    \delta_{\tau,x}(\varepsilon) =
    \sup_{\varepsilon_0+\ldots+\varepsilon_{n-1} < \varepsilon} \min \{\delta_{\tau_i,x}\circ\delta_{f,i}(\varepsilon_i)\mid i < n\}.
  \end{gather*}
  Here we follow the convention that $\delta(\infty) = \infty$. 
\item If $P$ is an $n$-ary predicate symbol and $\tau_0,\dots\tau_{n-1}$ are
  terms, then $P\tau_0\ldots\tau_{n-1}$ is a formula (called an \emph{atomic
    formula}).
  The definition of
  $\delta_{P\tau_0\ldots\tau_{n-1},x}$ is formally identical that of $\delta_{f\tau_0\ldots\tau_{n-1},x}$.
\item If $\varphi$ and $\psi$ are formulas then so are
  $\lnot\varphi$, $\half \varphi$ and $\varphi \dotplus \psi$.
  We have:
  \begin{align*}
    \delta_{\lnot\varphi,x}(\varepsilon) & = \delta_{\varphi,x}(\varepsilon) \\
    \delta_{\half \varphi,x}(\varepsilon) & = \delta_{\varphi,x}(2\varepsilon) \\
    \delta_{\varphi\dotplus \psi,x}(\varepsilon) & = \sup_{\varepsilon_0+\varepsilon_1 < \varepsilon} \min \{\delta_{\varphi,x}(\varepsilon_0),\delta_{\psi,x}(\varepsilon_1)\}.
  \end{align*}
\item Let $\Phi$ be a countable set of formulas in a finite tuple
  of free variables $\bar x$.
  For each variable $x$, let
  $\delta'_{\bigwedge\Phi,x} = \inf_{\varphi\in\Phi} \delta_{\varphi,x}\colon (0,\infty) \to [0,\infty]$.
  If $\delta'_{\bigwedge\Phi,x}(\varepsilon) > 0$ for all $\varepsilon > 0$ and $x \in \bar x$, then
  $\bigwedge\Phi$ is a formula, also denoted $\bigwedge_{\varphi\in\Phi} \varphi$.
  Its uniform continuity moduli are given by
  \[
  \delta_{\bigwedge\Phi,x}(\varepsilon) = \sup\{\delta'_{\bigwedge\Phi,x}(\varepsilon')\colon 0<\varepsilon'<\varepsilon\},
  \] so that \fref{eq:GoodContMod2} is satisfied.
\item If $\varphi$ is a formula and $x$ a variable, then
  $\inf_x \varphi$ is a formula.
  For $y\neq x$ we have $\delta_{\inf_x\varphi,y} = \delta_{\varphi,y}$, while
  $\delta_{\inf_x\varphi,x} = \infty$.
\end{itemize}

\end{definition}

\begin{notation}
  Rather than putting $\bigvee$ and $\sup$ in our language we define them as
  abbreviations:
  \begin{align*}
    \bigvee\Phi & := \lnot\bigwedge_{\varphi\in\Phi}\lnot\varphi \\
    \sup_x \varphi & := \lnot\inf_x\lnot\varphi.
  \end{align*}
\end{notation}

If $M$ is an $\cL$-structure and $\varphi(x_0,\dots,x_{n-1}) \in \cL_{\omega_1,\omega}$,
one constructs the interpretation $\varphi^M\colon M^n \to [0,1]$
in the obvious manner.
By induction on the structure of $\varphi$ one also shows that for each
variable $x$, $\varphi^M$ is uniformly continuous in $x$ respecting
$\delta_{\varphi,x}$.

Finitary continuous first order formulas, as defined in 
\cite{BenYaacov-Usvyatsov:localstability} and
\cite{BenYaacov-Berenstein-Henson-Usvyatsov:metricstructures},
are constructed in the
same manner, with the exclusion of the infinitary connectives $\bigwedge$ and
$\bigvee$ (i.e., only using the connectives $\lnot,\half,\dotplus$, or
equivalently $\lnot,\half,\dotminus$).
We observe that $\varphi\land\psi$ is equivalent to $\varphi \dotminus (\varphi \dotminus \psi)$,
so finitary instances of $\bigwedge$ and $\bigvee$ are allowed there as well.
The set of all such formulas is denoted $\cL_{\omega,\omega}$.

\begin{definition}
Let $\cL$ be a continuous signature and let $\phi$ be an
$\cL_{\omega_1,\omega}$-formula.
The set of \emph{subformulas} of  $\phi$ denoted $\sub(\phi)$, is defined inductively as follows.

\begin{itemize}

\item
If $P$ is a predicate symbol and $\tau_0,\dots\tau_{n-1}$ are terms, then $\sub(P\tau_0\ldots\tau_{n-1})=\{P\tau_0\ldots\tau_{n-1}\}$.

\item
$\sub(\lnot \phi)=\{\lnot\phi\}\cup\sub(\phi)$ and $\sub(\half \phi)=\{\half\phi\}\cup\sub(\phi)$.

\item
$\sub(\phi \dotplus \psi)=\{\,\phi \dotplus \psi\,\}\cup\sub(\phi)\cup\sub(\psi)$.

\item 
$\sub(\bigwedge_{\phi\in\Phi} \phi) = \{\,\bigwedge_{\phi\in\Phi} \phi\,\}\cup \bigcup_{\phi\in \Phi}\sub(\phi)$.

\item 
$\sub(\inf_x\phi) = \{\,\inf_x\phi\,\}\cup \sub(\phi)$.

\end{itemize}

\end{definition}

$\cL_{\omega_1,\omega}$ need not be countable if $\cL$ is countable. Nevertheless, it is often sufficient to work with countable fragments of $\cL_{\omega_1,\omega}$:

\begin{definition}
A \emph{fragment} of $\cL_{\omega_1,\omega}$ is subset of $\cL_{\omega_1,\omega}$
which contains  all atomic formulas and is closed under subformulas and
substitution of terms for free variables.
\end{definition}

\begin{remark}
Every countable subset of $\cL_{\omega_1,\omega}$ is contained in a countable fragment of $\cL_{\omega_1,\omega}$.
\end{remark}


For the next three sections (that is, the rest of the paper minus the last section), $\cL$ will denote a fixed countable continuous signature, and $\cL_A$ will denote a fixed countable fragment of $\cL_{\omega_1,\omega}$. We will let $C = \{c_i \mid i < \omega\}$ be a set of new constant symbols, and $\cL(C) = \cL\cup C$. An $\cL(C)$-structure $M$ will be called \emph{canonical} if the set $\{c_i^M \mid i < \omega\}$ is dense in $M$.

By $\cL_A(C)$ we will denote the smallest countable fragment of $\cL_{\omega_1,\omega}(C)$ that contains $\cL_A$; notice that $\cL_A(C)$ is obtained allowing closing $\cL_A$ under
substitution of constant symbols from $C$ for free variables.

We will also use the following notation:
  \begin{itemize}
  \item The set of all sentences in $\cL_A(C)$ will be denoted
    $\cL_A^s(C)$.
  \item The set of all atomic sentences in $\cL_A(C)$ will be denoted
    $\cL_A^{as}(C)$.
  \item The set of variable-free terms in $\cL(C)$ will be denoted
    $\cT(C)$.
  \end{itemize}

\section{Forcing}
\label{sec:forcing}

\begin{definition}
  A \emph{forcing property} for $\cL_A$ is a triplet
  $(\bP,\leq,f)$ where $(\bP,\leq)$ is a partially ordered set.
  The elements of $\bP$ are called \emph{conditions}.
  For each condition $p$, $f$ assigns a mapping
  $f_p \colon \cL_A^{as}(C) \to [0,1]$ satisfying the following conditions.
  \begin{enumerate}
  \item
    $p \leq q$ implies $f_p \leq f_q$ i.e., $f_p(\varphi) \leq f_q(\varphi)$ for all
    $\varphi \in \cL_A^{as}(C)$.
  \item Given $p\in \bP$, $\varepsilon>0$, $\tau,\sigma \in \cT(C)$, and an atomic
    $\cL(C)$-formula $\varphi(x)$ there are
    $q \leq p$ and $c \in C$ such that:
    \begin{gather*}
    f_q(d(\tau,c)) < \varepsilon,\\
    f_q(d(\tau,\sigma)) < f_p(d(\sigma,\tau)) + \varepsilon,\\
    \intertext{and if $f_p(d(\tau,\sigma)) < \delta_{\varphi,x}(\varepsilon)$,}
    f_q(\varphi(\sigma)) < f_p(\varphi(\tau)) + \varepsilon.
   	\end{gather*}
  \end{enumerate}
\end{definition}

For the rest of this section, $(\bP,\leq,f)$ will denote a fixed forcing property.

\begin{definition}
\label{D:forcing}
  Let $p \in \bP$ be a condition and $\varphi \in \cL_A^s(C)$ a sentence.
  We define $F_p(\varphi) \in [0,1]$ by induction on $\varphi$.
  For $\varphi$ atomic,
  \[
  F_p(\varphi) = f_p(\varphi).\\
  \]
  Otherwise,
  \[
  \begin{array}{lcl}
    F_p(\lnot\varphi) & = & \lnot \inf_{q\leq p} F_q(\varphi) \\
    F_p(\half \varphi) & = & \half F_p(\varphi) \\
    F_p(\varphi \dotplus \psi) & = & F_p(\varphi) \dotplus F_p(\psi) \\
    F_p(\bigwedge\Phi) & = & \inf_{\varphi\in\Phi} F_p(\varphi) \\
    F_p(\inf_x \varphi(x)) & = & \inf_{c\in C} F_p(\varphi(c)).
  \end{array}
  \]
  If $r \in \setR$ and $F_p(\varphi) < r$ we say that $p$ \emph{forces} that
  $\varphi < r$, in symbols $p \Vdash \varphi<r$.
\end{definition}

\begin{remark}
\label{R:forcing inequality properties}
  Let $p \in \bP$ be a condition, $\varphi \in \cL_A^s(C)$ a sentence, and
  $r \in \setR$. Then,
  \[
    \begin{array}{lcl}
    p \Vdash \varphi<r &\Longleftrightarrow & f_p(\varphi) < r, \text{ if $\varphi$ is atomic}\\
    p \Vdash \half \varphi < r & \Longleftrightarrow & p \Vdash \varphi < 2r\\
    p \Vdash \lnot\varphi < r & \Longleftrightarrow & (\exists s > 1-r)(\forall q \leq p)(q \nVdash \varphi < s)\\
    p \Vdash (\varphi \dotplus \psi) < r & \Longleftrightarrow &
      (\exists s)(p \Vdash \varphi < s  \text{ and } p \Vdash \psi < r-s)\\
    p \Vdash \bigwedge\Phi < r & \Longleftrightarrow & (\exists\varphi\in\Phi)(p \Vdash \varphi < r)\\
    p \Vdash \inf_x\varphi(x) < r & \Longleftrightarrow & (\exists c\in C)(p \Vdash \varphi(c) <r).
    \end{array}
  \]
\end{remark}

\begin{remark}
  The forcing relation $\Vdash$ can be defined inductively, without reference to the function $F_p(\phi)$, by the list of equivalences in the preceding remark. One can then define $F_p(\varphi)$ as $\inf\{\,r\in\mathbb{R}\mid p \Vdash \varphi<r \,\}$.
\end{remark}

The following basic properties will be used many times.

\begin{lemma}
  \label{lem:ForcingProps}
  For all $p,\varphi$,
  \begin{enumerate}
  \item $F_p(\varphi) \in [0,1]$.
  \item $q \leq p \Longrightarrow F_q(\varphi) \leq F_p(\varphi)$.
  \item $F_p(\varphi) + F_p(\lnot\varphi) \geq 1$.
  \end{enumerate}
\end{lemma}
\begin{proof}
  The first two items are by induction on the structure of $\varphi$.
  The last one follows directly from the definition.
\end{proof}

\begin{definition}
  We also define $F^w_p$ by:
  \begin{gather*}
    F^w_p(\varphi) = \sup_{q\leq p}\inf_{q'\leq q} F_{q'}(\varphi).
  \end{gather*}
  If $r \in \setR$ and $F^w_p(\varphi) < r$ we say that $p$ \emph{weakly forces}
  that $\varphi < r$, in symbols $p \Vdash^w \varphi<r$.
\end{definition}

By \fref{lem:ForcingProps}, $F^w_p(\varphi) \leq F_p(\varphi)$.

\begin{remark}
The weak forcing relation $\Vdash^w$ can be defined without reference to the function $F^w_p$ as follows:
  $p \Vdash^w \varphi<r$ if and only if
  $(\exists s < r)(\forall q\leq p)(\exists q'\leq q)(q' \Vdash \varphi<s)$. We can then define
  $F^w_p(\varphi)$ as $\inf\{\,r\mid p \Vdash^w \varphi<r \,\}$.
\end{remark}

\begin{lemma}
  \label{lem:WForceRobust}
  Let $p \in \bP$, $\varphi \in \cL_A^s(C)$ and $r \in \setR$.
  Then
  \begin{gather*}
    F^w_p(\varphi) = \sup_{q\leq p} F^w_q(\varphi) = \sup_{q\leq p}\inf_{q'\leq q} F^w_{q'}(\varphi).
  \end{gather*}
\end{lemma}
\begin{proof}
  That $F^w_p(\varphi) = \sup_{q\leq p} F^w_q(\varphi)$ follows easily from the
  definitions,
  and $\sup_{q\leq p} F^w_q(\varphi) \geq \sup_{q\leq p}\inf_{q'\leq q} F^w_{q'}(\varphi)$ is
  immediate.
  Finally:
  \begin{align*}
    \sup_{q\leq p}\inf_{q'\leq q} F^w_{q'}(\varphi)
    & = \sup_{q\leq p}\inf_{q'\leq q} \sup_{q''\leq q'} \inf_{q'''\leq q''} F_{q'''}(\varphi)
    \geq \sup_{q\leq p}\inf_{q'\leq q} \inf_{q'''\leq q'} F_{q'''}(\varphi) \\
    & = \sup_{q\leq p}\inf_{q'\leq q} F_{q'}(\varphi)
    = F^w_p(\varphi).
    \qedhere
  \end{align*}
\end{proof}

\begin{proposition}
  \label{prp:WForceInd}
  The weak forcing function $F^w$ obeys the following inductive rules:
  \[
  \begin{array}{lcl}
    F^w_p(\lnot\varphi) & = & \lnot \inf_{q\leq p} F^w_q(\varphi) \\
    F^w_p(\half \varphi) & = & \half F^w_p(\varphi) \\
    F^w_p(\varphi \dotplus \psi) & = & \sup_{q\leq p} \inf_{q'\leq q} F^w_{q'}(\varphi) \dotplus F^w_{q'}(\psi) \\
    F^w_p(\bigwedge\Phi) & = & \sup_{q\leq p} \inf_{q'\leq q} \inf_{\varphi\in\Phi} F^w_{q'}(\varphi) \\
    F^w_p(\inf_x \varphi(x)) & = & \sup_{q\leq p} \inf_{q'\leq q} \inf_{c\in C} F^w_{q'}(\varphi(c)).
  \end{array}
  \]
\end{proposition}
\begin{proof}
  For $\lnot\varphi$ and $\half\varphi$ this follows from a straightforward
  calculation.
  For example:
  \begin{align*}
    F^w_p(\lnot\varphi)
    & = \sup_{q\leq p}\inf_{q'\leq q} F_{q'}(\lnot\varphi)
    = \sup_{q\leq p}\inf_{q'\leq q} \lnot \inf_{q''\leq q'} F_{q'}(\varphi) \\
    & = \lnot \inf_{q\leq p}\sup_{q'\leq q} \inf_{q''\leq q'} F_{q'}(\varphi)
    = \lnot \inf_{q\leq p} F^w_q(\varphi),
  \end{align*}

  For the other three, the inequality $\geq$ is obtained substituting the
  definition of $F^w_p$ on the left hand side and using the fact that
  $F_p \geq F^w_p$.
  For $\leq$, we first use \fref{lem:WForceRobust} to replace each occurrence of
  $F^w_p$ on the left hand side with $\sup_{q<p}\inf_{q'\leq q} F^w_{q'}$. Thus, it will suffice to show that:
  \[
  \begin{array}{lcl}
    F^w_p(\varphi \dotplus \psi) & \leq & F^w_p(\varphi) \dotplus F^w_p(\psi) \\
    F^w_p(\bigwedge\Phi) & \leq & \inf_{\varphi\in\Phi} F^w_p(\varphi) \\
    F^w_p(\inf_x \varphi(x)) & \leq & \inf_{c\in C} F^w_p(\varphi(c)).
  \end{array}
  \]

  For $\dotplus$ assume $F^w_p(\varphi) = r$ and $F^w_p(\psi) = s$.
  Then for all $\varepsilon>0$ and for all $q \leq p$ there is $q'_0\leq q$ such that
  $F_{q'_0}(\varphi) < r+\varepsilon$, and as $q'_0 \leq p$ there is $q' \leq q'_0$ such that
  $F_{q'}(\psi) < s+\varepsilon$.
  Then $F_{q'}(\varphi \dotplus \psi) < r+s + 2\varepsilon$, yielding
  $F^w_{q'}(\varphi \dotplus \psi) \leq r+s$.

  For $\bigwedge\Phi$ and $\inf_x \varphi(x)$ it's a straightforward quantifier exchange
  argument, e.g.:
  \begin{align*}
    F^w_p(\bigwedge\Phi)
    & = \sup_{q\leq p}\inf_{q'\leq q} \inf_{\varphi\in\Phi} F_p(\varphi)
    \leq \inf_{\varphi\in\Phi} \sup_{q\leq p}\inf_{q'\leq q} F_p(\varphi) = \inf_{\varphi\in\Phi} F^w(\varphi).
    \qedhere
  \end{align*}
\end{proof}


\begin{lemma}
  \label{lem:WForceDense}
  For all $p \in \bP$ and $\tau$: $F^w_p(\inf_x d(\tau,x)) = 0$.
\end{lemma}
\begin{proof}
  If not then 
  $F^w_p(\inf_x d(\tau,x)) = \sup_{q\leq p} \inf_{q'\leq q} \inf_{c\in C}
  f_p(d(\tau,x)) > 0$.
  But this contradicts the
  definition of forcing property.
\end{proof}

\begin{definition}
	\label{dfn:Generic}
  A nonempty $G \subseteq \bP$ is \emph{generic} if:
  \begin{enumerate}
  \item It is directed downwards, i.e., for all $p,q \in G$ there is
    $p' \in G$ such that $p' \leq p,q$.
  \item It is closed upwards, i.e., if $p \in G$ and $q \geq p$ then $q \in G$.
  \item For every $\varphi \in \cL_A^s(C)$ and $r > 1$
    there is $p \in G$ such that $F_p(\varphi) + F_p(\lnot\varphi) < r$.
  \end{enumerate}
  If $G$ is a generic set and $\varphi \in \cL_A^s(C)$ we define
  $$\varphi^G = \inf_{p\in G} F_p(\varphi).$$
\end{definition}

\begin{proposition}
	\label{P:generic set}
  Every condition belongs to a generic set.
\end{proposition}
\begin{proof}
  Fix $p\in\bP$. Let $(\,(r_n,\varphi_n)\colon n < \omega\,)$ enumerate all pairs
  $(r,\varphi)$, where $r \in \setQ$, $r>1$, and $\varphi \in \cL_A^s(C)$.
  Construct a sequence $p_0 \geq p_1 \geq \ldots \geq p_n \geq \ldots$ in $\bP$ as
  follows.
  We start with $p_0 = p$.
  Assume $p_n$ has already been chosen.
  By definition $F_{p_n}(\lnot\varphi_n) + \inf_{q\leq p_n} F_{q}(\varphi_n) = 1 < r_n$, so
  we can choose $p_{n+1}\leq p_n$ such that
  $F_{p_n}(\lnot\varphi_n) + F_{p_{n+1}}(\varphi_n) < r_n$,
  whereby $F_{p_{n+1}}(\lnot\varphi_n) + F_{p_{n+1}}(\varphi_n) < r_n$.
  Define 
  $$G = \{\,q \in \bP\mid q \geq p_n\text{ for some }n\,\}.$$
  Then $G$ is generic, and $p \in G$.
\end{proof}

\begin{lemma}
  \label{lem:GenSetWForce}
  Let $G$ be generic and $\varphi \in \cL_A^s(C)$.
  Then $\varphi^G = \inf_{p\in G} F^w_p(\varphi)$.
\end{lemma}
\begin{proof}
  The inequality $\geq$ is immediate since $F^w_p(\phi)\leq F_p(\phi)$.
  For the other, assume $\varphi^G > \inf_{p\in G} F^w_p(\varphi)$, so there are
  $\varepsilon > 0$ and $p \in G$ such that $\varphi^G - \varepsilon > F^w_p(\varphi)$.
  As $G$ is generic there is $q \in G$ such that
  $F_q(\varphi) + F_q(\lnot\varphi) < 1+\varepsilon$, and as $p \in G$
  we may assume $q \leq p$.
  We obtain
  $$F^w_p(\varphi) \geq \inf_{q'\leq q}F_{q'}(\varphi) = 1-F_q(\lnot\varphi) > F_q(\varphi) - \varepsilon \geq \varphi^G - \varepsilon
  > F^w_p(\varphi),$$
  a contradiction.
\end{proof}

\begin{lemma}
  \label{lem:GenSetNeg}
  If $G$ is generic and $\varphi \in \cL_A^s(C)$, then
  $(\lnot\varphi)^G = 1-\varphi^G$.
\end{lemma}
\begin{proof}
  From \fref{lem:ForcingProps} we have $\varphi^G+(\lnot\varphi)^G \geq 1$, while
  $\varphi^G+(\lnot\varphi)^G \leq 1$ follows from \fref{dfn:Generic}.
\end{proof}

\begin{lemma}
  \label{lem:GenSetMetric}
  Let $G$ be a generic set and $\tau,\sigma \in \cT(C)$.
  Then:
  \begin{enumerate}
  \item For every $\varepsilon > 0$ there is $c_{\tau,\varepsilon,G} \in C$ such that
    $d(\tau,c_{\tau,\varepsilon,G})^G < \varepsilon$.
  \item $d(\tau,\sigma)^G = d(\sigma,\tau)^G$.
  \item For every atomic $\cL(C)$-formula $\varphi(x)$,
    if $d(\tau,\sigma)^G < \delta_{\varphi,x}(\varepsilon)$ then $|\varphi(\tau)^G-\varphi(\sigma)^G| < \varepsilon$.
  \end{enumerate}
\end{lemma}
\begin{proof}
  For (i), observe that $(\inf_x d(\tau,x))^G =0$ by
  \fref{lem:WForceDense} and \fref{lem:GenSetWForce}, so there is $p \in G$ such that
  $F_p(\inf_x d(\tau,x)) < \varepsilon$, and thus there exists $c \in C$ such
  that $d(\tau,c)^G \leq F_p(d(\tau,c)) < \varepsilon$.
  The other two statements follow directly from
  \fref{lem:GenSetWForce} and the definition of forcing property.
\end{proof}

\begin{lemma}
	\label{L:generic model}
  Let $M_0^G$ be the term algebra $\cT(C)$ equipped with the natural
  interpretation of the function symbols, and interpreting the
  predicate symbols by:   $P^{M_0^G}(\bar \tau) = P(\bar \tau)^G$.
  Then $M_0^G$ is a pre-$\cL(C)$-structure, and its completion $M^G$ is a
  canonical structure.
\end{lemma}
\begin{proof}
  First we use \fref{lem:GenSetMetric} to show that $d^{M_0^G}$ is a pseudometric.
  Symmetry is \fref{lem:GenSetMetric}(ii).
  The triangle inequality follows from \fref{lem:GenSetMetric}(iii),
  keeping in mind that $\delta_{d(x,\sigma),x} = \id$.
  That $d^{M_0^G}(\tau,\tau) = 0$ follows from the
  triangle inequality and \fref{lem:GenSetMetric}(i).
  Finally, by \fref{lem:GenSetMetric}(iii),
  every symbol respects its
  uniform continuity modulus.
  Thus $M_0^G$ is a pre-structure, and we can define $M^G$ to be its
  completion.

  That $C^{M^G}$ is dense in $M^G$ now follows from
  \fref{lem:GenSetMetric}(i).
\end{proof}

\begin{theorem}
	\label{T:generic model}
  For all $\varphi \in \cL_A^s(C)$ we have $\varphi^{M^G} = \varphi^G$.
\end{theorem}
\begin{proof}
  By induction on $\varphi$:
  \begin{enumerate}
  \item For $\varphi$ atomic, this is immediate from the construction of
    $M^G$.
  \item For $\half \varphi$, $\varphi \dotplus \psi$ and $\bigwedge\Phi$, this is immediate from
    the definition of forcing and the induction hypothesis.
  \item For $\lnot\varphi$, this is immediate from \fref{lem:GenSetNeg}
    and the induction hypothesis.
  \item For $\inf_x \varphi(x)$, it follows from the definition of forcing
    and the induction hypothesis that
    $(\inf_x \varphi)^G = \inf\{\varphi(c)^{M^G}\mid c \in C\}$.
    Since $C^{M^G}$ is dense in $M^G$ and $\varphi(x)^{M^G}$ is uniformly
    continuous in $x$, the latter is equal to $(\inf_x \varphi)^{M^G}$.
    \qedhere
  \end{enumerate}
\end{proof}

\section{The forcing Properties   $\mathcal{P}(\mathcal{M})$ and $\mathcal{P}(\mathcal{M},\Sigma)$}
\label{sec:special forcing properties}

If $\mathcal{M}$ is class of $\cL$-structures, we denote by $\mathcal{M}(C)$ the class of all structures of the form $(M, a_c)_{c\in C_0}$, where $M$ is in $\mathcal{M}$  and $C_0$ is a finite subset of $C$; such a structure is regarded naturally as an $\cL(C_0)$-structure by letting $a_c$ be the interpretation of $c$ in $M$, for each $c\in C_0$.

Let $\Sigma$ be a class of formulas of $\cL_A$ that contains all the atomic formulas and is closed under subformulas, and let $\Sigma(C)$ denote the subset of $\cL(C)$ obtained from formulas $\phi$ in $\Sigma$ by replacing finitely many free variables of $\phi$ with constant symbols from $C$.

The forcing property $\mathcal{P}(\mathcal{M},\Sigma)$ is defined as follows. The conditions of $\mathcal{P}(\mathcal{M},\Sigma)$ are the finite sets of the form
\[
\{\,\phi_1<r_1,\dots,\phi_n<r_n\,\},
\]
where $\phi_1,\dots,\phi_n\in\Sigma(C)$ and there exist $M\in \mathcal{M}(C)$ such that $\phi_i^M<r_i$, for $i=1,\dots, n$. The partial order $\leq$ on conditions is reverse inclusion, i.e., if $p,q$ are conditions of $\mathcal{P}(\mathcal{M},\Sigma)$, then $p\leq q$ if and only $p\supseteq q$. If $p$ is a condition of  $\mathcal{P}(\mathcal{M}_\Delta,\Sigma)$ and $\phi$ is an atomic sentence of $\cL(C)$, we define
\[
f_p(\phi)=
\begin{cases}
\min \{r\leq 1 \mid \phi<r\in p\} ,\quad &\text{if $\{r\leq 1 \mid \phi<r\in p\}\neq \emptyset$,}\\
1,\quad&\text{otherwise}.
\end{cases}
\]

When $\Sigma$ is the set of all atomic $\cL$-formulas, the forcing property $\mathcal{P}(\mathcal{M},\Sigma)$  is denoted simply $\mathcal{P}(\mathcal{M})$.

The main result of this section is Proposition~\ref{P:w-forcing in P(M,Sigma)}, below, which characterizes weak forcing for the forcing property $\mathcal{P}(\mathcal{M},\Sigma)$; for the proof, we need two lemmas.

\begin{definition}
  We extend the definition of $f_p$ above to all sentences of $\Sigma(C)$:
  \[
  H_p(\varphi) =
  \begin{cases}
    \min \{r\leq 1 \mid \phi<r\in p\} ,\quad &\text{if $\{r\leq 1 \mid \phi<r\in p\}\neq \emptyset$,}\\
    1,\quad&\text{otherwise}.
  \end{cases}
  \]
  We define $H^w_p$ accordingly:
  $H^w_p(\varphi) = \sup_{q\leq p} \inf_{p\leq q} H_p(\varphi)$.
\end{definition}

Clearly if $q \leq p$ then $H_q(\varphi) \leq H_p(\varphi)$
and $H^w_q(\varphi) \leq H^w_p(\varphi)$, whereby for all $p$:
$H^w_p(\varphi) \leq H_p(\varphi)$.

\begin{lemma}
  \label{lem:WPMSChar}
  For all $p \in \PMS$ and $\varphi \in \Sigma(C)$:
  \begin{align*}
    H^w_p(\varphi) & = \inf \{r \in [0,1] \mid (\forall q \leq p)(q \cup \{\varphi<r\} \in \PMS)\} \\
    & = \sup \{r \in [0,1]\mid p\cup\{\lnot\varphi < 1-r\} \in \PMS\}
  \end{align*}
  (Here $\inf \emptyset = 1$, $\sup \emptyset = 0$.)
\end{lemma}
\begin{proof}
  The first equality is a mere rephrasing:
  $H^w_p(\varphi) \leq r$ if and only if $\inf_{q'\leq q} H_p(\varphi) \leq r$
  for all $q \leq p$, i.e., if and only if
  $q \cup \{\varphi<r\} \in \PMS$ for all $q \leq p$.

  For the second equality:
  Assume first that $q = p\cup\{\lnot\varphi < 1-r\} \in \PMS$.
  Then $q \leq p$ but $q \cup \{\varphi<r\} \notin \PMS$.
  This gives $\geq$.
  Now assume $p\cup\{\lnot\varphi < 1-r\} \notin \PMS$.
  Then $p\cup\{\lnot\varphi < 1-r\}$ cannot be realized in
  the given class.
  Thus, for every $q \leq p$, as $q$ can be realized, it is realized in a
  model where $\varphi \leq r$.
  Thus $q \cup \{\varphi<s\} \in \PMS$ for all $q \leq p$ and $s > r$.
  This gives $\leq$.
\end{proof}

\begin{proposition}
  \label{prp:WPMSInd}
  The functions $H^w_p$ satisfy the properties stated for
  $F^w_p$ in \fref{lem:WForceRobust} and \fref{prp:WForceInd}, i.e.:
  \[
  \begin{array}{lcl}
    H^w_p(\varphi) & = & \sup_{q\leq p} H^w_q(\varphi)
    = \sup_{q\leq p}\inf_{q'\leq q} H^w_{q'}(\varphi) \\
    H^w_p(\lnot\varphi) & = & \lnot \inf_{q\leq p} H^w_q(\varphi) \\
    H^w_p(\half \varphi) & = & \half H^w_p(\varphi) \\
    H^w_p(\varphi \dotplus \psi) & = & \sup_{q\leq p} \inf_{q'\leq q} H^w_{q'}(\varphi) \dotplus H^w_{q'}(\psi) \\
    H^w_p(\bigwedge\Phi) & = & \sup_{q\leq p} \inf_{q'\leq q} \inf_{\varphi\in\Phi} H^w_{q'}(\varphi) \\
    H^w_p(\inf_x \varphi(x)) & = & \sup_{q\leq p} \inf_{q'\leq q} \inf_{c\in C} H^w_{q'}(\varphi(c)).
  \end{array}
  \]
\end{proposition}
\begin{proof}
  The first property is proved precisely as in
  \fref{lem:WForceRobust}.

  For  $\lnot$: it follows from \fref{lem:WPMSChar} that
  $H^w_p(\lnot\varphi) = \lnot\inf_{q\leq p}H_q(\varphi)$, and we conclude as in
  the proof of \fref{prp:WForceInd}.

  For $\half$: observe that
  $q\cup\{\varphi<r\} \in \PMS$ if and only if
  $q\cup\{\half \varphi < \half r\} \in \PMS$ and apply
  \fref{lem:WPMSChar}.

  For the last three we reduce as in the proof of
  \fref{prp:WForceInd} to showing that:
  \[
  \begin{array}{lcl}
    H^w_p(\varphi \dotplus \psi) & \leq & H^w_p(\varphi) \dotplus H^w_p(\psi) \\
    H^w_p(\bigwedge\Phi) & \leq & \inf_{\varphi\in\Phi} H^w_p(\varphi) \\
    H^w_p(\inf_x \varphi(x)) & \leq & \inf_{c\in C} H^w_p(\varphi(c)).
  \end{array}
  \]
  For $\dotplus$ this follows from \fref{lem:WPMSChar}.
  For $\bigwedge$ and $\inf$ the quantifier exchange argument from
  proof of the corresponding items
  in \fref{prp:WForceInd} works here too.
\end{proof}

\begin{proposition}
\label{P:w-forcing in P(M,Sigma)}
Suppose that $p$ is a condition in the forcing property $\mathcal{P}(\mathcal{M},\Sigma)$ and 
$\sigma$ is a sentence of $\Sigma(C)$. Then $F^w_p(\sigma) = H^w_p(\sigma)$.
\end{proposition}
 
\begin{proof}
  For atomic $\sigma$ the equality is immediate.
  We then proceed by induction on $\sigma$, noting that
  \fref{prp:WForceInd} on the one hand and \fref{prp:WPMSInd} on the
  other tell us that $F^w_p$ and $H^w_p$ obey the same inductive definitions.
\end{proof}

\section{Generic Models and \sci-Formulas}
\label{sec:omitting types}

Recall from \fref{sec:preliminaries} that the expressions $\bigvee\Phi$ and $\sup_x \varphi$ are regarded abbreviations of $\lnot\bigwedge_{\varphi\in\Phi}\lnot\varphi$ and
$\lnot\inf_x\lnot\varphi$ respectively.

\begin{proposition}
\label{P:disjunctions}
Let $(\bP,\leq,f)$ be a forcing property for $\cL_A(C)$ and let $p\in \bP$.
Then
\begin{enumerate}
\item
$F_p(\bigvee\Phi)=\sup_{\phi\in\Phi}F_p^w(\phi)$.
\item
$F_p(\sup_x\phi(x))=\sup_{c\in C} F_p^w(\phi(c))$.
\end{enumerate}
\end{proposition}

\begin{proof}
The proofs are straightforward applications of the definitions: for~(i),
\begin{align*}
F_p(\bigvee \Phi)= F_p(\lnot\bigwedge_{\phi\in\Phi} \lnot\phi)
&=\lnot \inf_{q\leq p}F_q(\bigwedge_{\phi\in\Phi}\lnot\phi)\\
&=\lnot \inf_{q\leq p}\inf_{\phi\in\Phi}F_q(\lnot\phi)\\
&=\lnot \inf_{q\leq p}\inf_{\phi\in\Phi}\lnot\inf_{q'\leq q} F_{q'}(\phi)\\
&=\sup_{q\leq p}\sup_{\phi\in\Phi}\inf_{q'\leq q} F_{q'}(\phi)\\
&=\sup_{\phi\in\Phi}\sup_{q\leq p}\inf_{q'\leq q} F_{q'}(\phi)\\
&=\sup_{\phi\in\Phi}F_p^w(\phi),
\end{align*}
and for~(ii),
\begin{align*}
F_p(\sup_x\phi(x))= F_p(\lnot\inf_x \lnot\phi(x))&=\lnot\inf_{q\leq p}F_q(\inf_x \lnot\phi(x))\\
&=\lnot\inf_{q\leq p}\inf_{c\in C}F_q(\lnot\phi(c))\\
&=\lnot\inf_{q\leq p}\inf_{c\in C}\lnot\inf_{q'\leq q}F_{q'}(\phi(c))\\
&=\sup_{q\leq p}\sup_{c\in C}\inf_{q'\leq q}F_{q'}(\phi(c))\\
&=\sup_{c\in C}\sup_{q\leq p}\inf_{q'\leq q}F_{q'}(\phi(c))\\
&=\sup_{c\in C}F_{q'}^w(\phi(c)).
\qedhere
\end{align*}
\end{proof}

\begin{notation}
If $\Phi$ is a finite set of formulas, say $\Phi=\{\phi_1,\dots,\phi_n\}$, we write
\[
\phi_1\land\dots\land\phi_n\quad\text{and}\quad \phi_1\lor\dots\lor\phi_n
\]
as abbreviations of $\bigwedge\Phi$ and $\bigvee\Phi$, respectively.
\end{notation}

\begin{proposition}
\label{P:finite disjunctions}
If $(\bP,\leq,f)$ is a forcing property for $\cL_A(C)$ and $p\in \mathbb{P}$, then
\[
F_p(\phi_1\lor\dots\lor\phi_n)=\max_{i}F_p^w(\phi_i).
\]

\end{proposition}

\begin{proof}
By Proposition~\ref{P:disjunctions}.
\end{proof}

\begin{definition}
  Let $\Sigma$ be a class of formulas of $\cL_A$ which contains all atomic
  formulas and is closed under subformulas.
  A \emph{$\sup\bigwedge\inf$-formula over $\Sigma$} is an $\cL_A$-formula of the form
\[
\sup_{x_1}\dots\sup_{x_m}\bigwedge_{n<\omega}\inf_{y_1}\dots \inf_{y_{i(n)}}
(\sigma_{n,1}(\bar x,\bar y_n)\lor \dots \lor \sigma_{n,j(n)}(\bar x,\bar y_n)),
\]
where $\sigma_{n,\nu}$ belongs to $\Sigma$ for $n<\omega$ and $\nu=1,\dots,j(n)$, $\bar x=x_1,\dots,x_m$, and $\bar y_n=y_1,\dots, y_{i(n)}$.
\end{definition}

\begin{proposition}
	\label{P:forcing special formulas}
Let $\Sigma$ be a class of formulas of $\cL_A$ which contains all atomic formulas and is closed under subformulas.
Suppose that $\phi$ is a $\sup\bigwedge\inf$-formula over $\Sigma$, of the form
\[
\sup_{x_1}\dots\sup_{x_m}\bigwedge_{n<\omega}\inf_{y_1}\dots \inf_{y_{i(n)}}
(\sigma_{n,1}(\bar x,\bar y_n)\lor \dots \lor \sigma_{n,j(n)}(\bar x,\bar y_n)),
\]
where $\sigma_{n,\nu}$ belongs to $\Sigma$ for $n<\omega$ and $\nu=1,\dots,j(n)$, $\bar x=x_1,\dots,x_m$, and $\bar y_n=y_1,\dots, y_{i(n)}$.
Then, if $(\bP,\leq,f)$ is a forcing property for $\cL_A(C)$ and $p\in \mathbb{P}$, 
\[
F_p(\phi)=
\sup_{\substack{\bar c\in C^m\\q\leq p}}
\inf_{\substack{q'\leq q\\ \bar d\in C^{i(n)}\\ n<\omega}}
\max_{1\leq\nu\leq j(n)}
F_{q'}^w(\, \sigma_{n,\nu}(\bar c,\bar d)\,).
\]
\end{proposition}

\begin{proof}
We use Propositions~\ref{P:disjunctions} and \ref{P:finite disjunctions} to compute $F_p(\phi)$:
\begin{align*}
&F_p(\,\sup_{x_1}\dots\sup_{x_m}\bigwedge_{n<\omega}\inf_{y_1}\dots \inf_{y_{i(n)}}
(\sigma_{n,1}(\bar x,\bar y_n)\lor \dots \lor \sigma_{n,j(n)}(\bar x,\bar y_n)\, )) & \\
=&\sup_{\bar c\in C^m} 
F_p^w(\, \bigwedge_{n<\omega}\inf_{y_1}\dots \inf_{y_{i(n)}}
(\sigma_{n,1}(\bar c,\bar y_n)\lor \dots \lor \sigma_{n,j(n)}(\bar c,\bar y_n))\,) &\text{(by \ref{P:disjunctions})}\\
=&\sup_{\bar c\in C^m} \sup_{q\leq p}\inf_{q'\leq q}
F_{q'}(\, \bigwedge_{n<\omega}\inf_{y_1}\dots \inf_{y_{i(n)}}
(\sigma_{n,1}(\bar c,\bar y_n)\lor \dots \lor \sigma_{n,j(n)}(\bar c,\bar y_n))\,) & \\
=&\sup_{\bar c\in C^m} \sup_{q\leq p}\inf_{q'\leq q}\inf_{n<\omega}\inf_{\bar d\in C^{i(n)}}
F_{q'}(\,\sigma_{n,1}(\bar c,\bar d)\lor \dots \lor \sigma_{n,j(n)}(\bar c,\bar d)\,) & \\
=&\sup_{\bar c\in C^m} \sup_{q\leq p}\inf_{q'\leq q}\inf_{n<\omega}\inf_{\bar d\in C^{i(n)}}\max_{1\leq\nu\leq j(n)}
F_{q'}^w(\, \sigma_{n,\nu}(\bar c,\bar d)\,)&\text{(by \ref{P:finite disjunctions})}.
\end{align*}
\end{proof}

\begin{remark}
\label{R:forcing special formulas}
If $p$ is a condition in the forcing property $\mathcal{P}(\mathcal{M},\Sigma)$, then $F_p^w(\phi)=H_p^w(\phi)$, by Proposition~\ref{P:w-forcing in P(M,Sigma)}. Hence, if $\phi$ is as in Proposition~\ref{P:forcing special formulas},
\[
F_p(\phi)=
\sup_{\substack{\bar c\in C^m\\q\leq p}}
\inf_{\substack{q'\leq q\\ \bar d\in C^{i(n)}\\ n<\omega}}
\max_{1\leq\nu\leq j(n)}
H_{q'}^w(\, \sigma_{n,\nu}(\bar c,\bar d)\,).
\]
\end{remark}

Recall (\fref{sec:preliminaries}) that $C$ denotes countable set of constants not in $\cL$ and that $\cL(C)=\cL\cup C$. As in Section~\ref{sec:special forcing properties}, if $\mathcal{M}$ is a class of $\cL$-structures, $\mathcal{M}(C)$ denotes the class of structures of the form $(M, a_c)_{c\in C_0}$, where $M$ is in $\mathcal{M}$  and $C_0$ is a finite subset of $C$.  If $\Gamma$ is a set of inequalities of the form $\phi<r$, where  $\phi$ is an $\cL_A(C)$-formula and $r$ is a real number, we will say that $\Gamma$ is \emph{satisfiable} in $\mathcal{M}$ if there exists a structure $M$ in $\mathcal{M}(C)$ such that $\phi^M<r$ for every inequality $\phi<r$ in $\Gamma$. 

Let $\Sigma$ be class of formulas of $\cL_A$ that contains all the atomic formulas and is closed under subformulas. A \emph{finite $\Sigma$-piece of $\mathcal{M}$} is a finite set $p$ of inequalities of the form $\phi<r$, where $\phi\in\Sigma$, such that $p$ that is satisfiable in $\mathcal{M}$.

\begin{corollary}[Omitting Types Theorem]
  \label{cor:omitting types}
  Let $(\,\phi_n\mid n<\omega\,)$ be a sequence of $\cL$-formulas such that for
  each $n<\omega$ $\phi_n$ is a $\sup\bigwedge\inf$-formula over $\Sigma$, of the form
  \begin{gather*}
    \sup_{x_1}\dots\sup_{x_{m(n)}}\psi_n(x_1,\dots,x_{m(n)}),
  \end{gather*}
  where for each $n<\omega$ $\psi_n$ is of the form
  \[
  \bigwedge_{k<\omega}\inf_{y_1}\dots \inf_{y_{i(n,k)}}
  (\sigma_{n,k,1}(\bar x_n,\bar y_{n,k})\lor \dots \lor \sigma_{n,k,j(n,k)}(\bar x_n,\bar y_{n,k})),
  \]
  with $\bar x_n=x_1,\dots,x_{m(n)}$, and
  $\bar y_{n,k}=y_1,\dots,y_{i(n,k)}, $ and let  $(\,r_n\mid n<\omega\,)$
  be a sequence of real numbers such that for every finite
  $\Sigma$-piece $p$ of $\mathcal{M}$ and every $\bar c_n\in C^{m(n)}$,
  the set  $p\cup\{\psi_n(\bar c_n)<r_n\}$ is satisfiable in $\mathcal{M}$.
  Then there exists a canonical $\cL(C)$-structure $M$ such that
  $\phi_n^M\leq r_n$ for every $n<\omega$.
\end{corollary}
\begin{proof}
  Let $p$ be a condition  in the forcing property
  $\mathcal{P}(\mathcal{M},\Sigma)$.
  Fix a condition $q \leq p$,  $n<\omega$ and $\bar c_n\in C^{m(n)}$. 
Since $q\cup\{\psi_n(\bar c_n)<r_n\}$ is satisfiable in $\mathcal{M}$, there exist $k<\omega$ and $\bar d_{n,k}\in C^{i(n,k)}$ such that
\[
q\cup\{\, \sigma_{n,k,1}(\bar c_n,\bar d_{n,k})\lor \dots \lor \sigma_{n,k,j(n,k)}(\bar c_n,\bar d_{n,k})<r_n \,\}
\]
is satisfiable in $\mathcal{M}$. Let
\[
q'=q\cup\{\, \sigma_{n,k,1}(\bar c_n,\bar d_{n,k})<r_n, \dots, \sigma_{n,k,j(n,k)}(\bar c_n,\bar d_{n,k})<r_n \,\}.
\]
Then, $q$ is a condition in $\mathcal{P}(\mathcal{M},\Sigma)$, and by \fref{lem:WPMSChar},
\[
\max_{1\leq\nu\leq j(n,k)}H_{q'}^w(\, \sigma_{n,k,\nu}(\bar c_n,\bar d_{n,k})\,)< r_n,
\]
so
\[
\inf_{\substack{q'\leq q\\ \bar d\in C^{i(n,k)}\\ k<\omega}}
\max_{1\leq\nu\leq j(n,k)}
H_{q'}^w(\, \sigma_{n,k,\nu}(\bar c_n,\bar d_{n,k}) \,)< r_n.
\]
Thus, by Remark~\ref{R:forcing special formulas}, $F_p(\phi_n)\leq r_n$. Let $G$ be a generic set for $\mathcal{P}(\mathcal{M},\Sigma)$ (the existence of $G$ is guaranteed by Proposition~\ref{P:generic set}). For every $n<\omega$, $\phi_n^G=\inf_{p\in G} F_p(\phi_n)\leq r_n$ (see Definition~\ref{dfn:Generic}). Let now $M^G$ be as in Lemma~\ref{L:generic model}. Then, by Theorem~\ref{T:generic model}, $\phi_n^{M^G}=\phi_n^G\leq r_n$.
\end{proof}

\begin{remark}
  The reader may worry about the fact that the assumptions of
  \fref{cor:omitting types} involve strict inequalities while the
  conclusion only yields a weak inequality.
  In fact it would be enough to assume a weak inequality,
  i.e., that
  $p \cup \{ \psi_n(\bar c_n) \leq r_n\}$ is satisfiable in $\cM$ for every
  $p$ and $\psi_n$ as in the statement of \fref{cor:omitting types},
  or more precisely, that $p \cup \{ \psi_n(\bar c_n) < r_n + \varepsilon\}$ is
  satisfiable for every $\varepsilon > 0$.
  Indeed, in this case we would be able to find $M$ in which
  $\varphi_n^M \leq r_n + 2^{-m}$ for all $n,m$, i.e., such that
  $\varphi_n^M \leq r_n$ for all $n$.
\end{remark}

\section{Application: Separable Quotients}
\label{sec:separable quotients}

If $\phi$ is a formula of $\cL_{\omega_1,\omega}$, we will say that $\phi$ is
\emph{finitary} if all the occurrences of $\bigwedge$ in $\phi$ are finitary,
i.e., if whenever $\bigwedge_{\psi\in\Phi} \psi$ is a subformula of $\phi$, the set $\Phi$ is
finite.
We recall from \fref{sec:preliminaries} that the set of all finitary
formulas is denoted $\cL_{\omega,\omega}$.

If $M$ and $N$ are $\cL$-structures, $M$ and $N$ are said to be \emph{elementary equivalent}, written $M\equiv N$, if $\phi^M=\phi^N$ for every finitary $\cL$-sentence $\phi$. Thus $M\equiv N$ if and only if  $\phi^M<r$ implies $\phi^N<r$ for every finitary $\cL$-sentence $\phi$ and every rational number $r$. If $M$ is a substructure of $N$, $M$ is said to be an \emph{elementary substructure} of $N$ if $(M,a \mid a\in M)\equiv(N,a\mid a\in M)$.

A Banach space $(X,\|\cdot\|)$ can be regarded as a metric structure in a number of ways. A natural approach is to introduce for each nonnegative rational $r$ a distinct sort for the closed ball $B_X(r)$ of radius $r$ around $0$; the metric on $B_X(r)$ is given by the norm $\|\cdot\|$; in the structure we also include:
\begin{itemize}
\item
  the inclusion maps $I_{r,s}\colon B_X(r)\to B_X(s)$ for $r<s$,
\item
  the vector addition, which maps $B_X(r)\times B_X(s)$ onto $B_X(r+s)$,
\item
  for each $\lambda\in\mathbb{Q}$, the scalar multiplication by $\lambda$, which maps $B_X(r)$ onto $B_X(|\lambda|r)$,
\item
  the normalized norm predicate $\|\cdot\|/r$, which maps $B_X(r)$ onto the
  interval $[0,1]$,
\item the normalized distance predicate on $B_X(r)$ defined by
  $d(x,y) = \|x-y\|/(2r)$ (as $x-y \in B_X(2r)$).
\end{itemize}

Notice that with the normalized norm and distance, all symbols are
$1$-Lipschitz, meaning that the identity function $\delta(\epsilon)=\epsilon$ is a
modulus of uniform continuity for each and every one of them.

Other ways of regarding Banach space as metric structures are
discussed in Section~3 of~\cite{BenYaacov-Usvyatsov:localstability}
and in Section~4 \cite{BenYaacov:Perturbations}.

If $X$ and $Y$ are Banach spaces and $T\colon X\to Y$ is a Banach space
operator, we denote by $(X,Y,T)$ the structure that
includes, in addition to the Banach space structure of $X$
and the Banach space structure on $Y$, in separate sorts,  the operator $T$ as
a family of functions between the appropriate sorts, i.e.,
from $B_X(r)$ to $B_Y(s)$ if $s \geq \|T\|r$.
If $T$ is nonzero, the function $\delta(\epsilon)=\epsilon/\|T\|$ is a modulus
of uniform continuity for $T$.

\begin{proposition}
\label{P:surjectivity}
Let $X,Y,\Hat{X},\Hat{Y}$ be Banach spaces and let $T\colon X\to Y$ and $\Hat{T}\colon\Hat{X}\to\Hat{Y}$ be bounded linear operators such that $(X,Y,T)\equiv(\Hat{X},\Hat{Y},\Hat{T})$. Then $T$ is surjective if and only if $\Hat{T}$ is surjective.
\end{proposition}

\begin{proof}
For a real number $r\geq 0$, let $B_X(r)$ and $B_Y(r)$ denote the closed balls of radius $r$ around $0$ in $X$ and $Y$, respectively. The proof of the Open Mapping Theorem shows that $T\colon X\to Y$ is surjective if and only if the following holds: for every $\epsilon>0$ there exists $\delta(\epsilon)>0$ such that
\[
\forall y\in B_Y(\delta(\epsilon)) \ \exists x\in B_X(1) \ (\,\|T(x)-y\|\leq \epsilon \,).
\]
We can assume that $\delta(\epsilon)<1$ for $\epsilon<1$. Thus, $T$ is surjective if and only if for every $\epsilon$ with $0<\epsilon<1$ we have $\phi_\epsilon^{(X,Y,T)}=0$, where $\phi_\epsilon$ is the following finitary sentence (the variables $x$ and $y$ are of sort $B_X(1)$ and $B_Y(1)$, respectively):
\[
\sup_y \ 
\inf_x \ \min\big(\ \delta(\epsilon)\dotminus \|y\|,\|T(x)- y\|\dotminus \epsilon \ \big).
\]
(Or, if one wishes to be pedantic, replace
$\|T(x)- y\|\dotminus \epsilon$ with $\frac{1}{2}\|T(x)-y\| \dotminus \frac{\epsilon}{2}$.)
Hence, if $(X,Y,T)\equiv(\Hat{X},\Hat{Y},\Hat{T})$ and $T$ is surjective, $\Hat{T}$ is surjective too.
\end{proof}

The authors are grateful to William B.~Johnson for pointing out that Proposition~\ref{P:surjectivity} is given by the proof of the Open Mapping Theorem.

All the Banach spaces mentioned henceforth will be infinite dimensional.

The Separable Quotient Problem is perhaps the most prominent open problem in nonseparable Banach space theory. The question is whether for every nonseparable Banach space $X$ there exist a separable Banach space $Y$ and a surjective operator $T\colon X\to Y$. Let $T\colon X\to Y$ be a Banach space operator and consider the structure $(X, Y, T)$ (the sorts of this structure are $X$ and $Y$).  In this section we use Corollary~\ref{cor:omitting types} to prove that if $T\colon X\to Y$ is surjective and has infinite dimensional kernel, then there exists an operator $\Hat{T}\colon \Hat{X}\to \Hat{Y}$ such that
\begin{enumerate}
\item
$(X,Y,T)\equiv(\Hat{X},\Hat{Y},\Hat{T})$,
\item
$\Hat{X}$ has density character $\omega_1$,
\item
$\Hat{Y}$ is separable. 
\end{enumerate}
It follows from (i) and Proposition~\ref{P:surjectivity} that  $\Hat{T}$ is surjective.

\begin{lemma}
  \label{L:Hahn-Banach}
  Suppose that $X$ is a Banach space and $Y$ is a	closed proper subspace
  of $X$.
  Then there exists a non-zero linear functional $f\colon X\to\mathbb{R}$
  whose restriction to $Y$ is zero.
  Up to multiplication by a scalar we may further assume that
  $\|f\| = 1$.
\end{lemma}
\begin{proof}
  This is a well-known application of the Hahn-Banach theorem;  the proof  can be found in a textbook, e.g.,~\cite{Fabian-Habala-Hajek-Montesinos-Pelant-Zizler:2002}.
\end{proof}

\begin{lemma}
  \label{L:sequence}
  If $X$ is a Banach space of density character $\kappa$, there exists a
  family $(x_i)_{i<\kappa}$ in $X$ such that $\|x_i\|=1$ for every $i<\kappa$ and
  $\|x_i-x_j\|\geq 1$ for $i<j<\kappa$.
\end{lemma}
\begin{proof}
  The construction of  $(x_i)_{i<\kappa}$ is inductive.
  Fix $j<\kappa$  and suppose that constructed $x_i$ is defined for $i<j$.
  Let $Y$ be the closed linear span of $\{ x_i \mid i<j\}$, which is
  a proper closed subspace of $X$.
  By Lemma~\ref{L:Hahn-Banach}, take $f\colon X\to\mathbb{R}$ such that $f(x)=0$ for every $x\in Y$ and $\|f\|=1$. Let now $x_j$ be an element of the unit sphere of $X$ such that $|f(x)|=\|f\|=1$. Then, if $y\in Y$, we have $\|x_j-y\|\geq |f(x_j)-f(y)|=1$; in particular, $\|x_j-x_i\|\geq 1$ for $i<j$.
\end{proof}

\begin{theorem}
  \label{T:separable quotient}
  For every surjective operator $T\colon X\to Y$ with infinite dimensional kernel there exists an operator $\Hat{T}\colon \Hat{X}\to \Hat{Y}$ such that
  \begin{itemize}
  \item
    $(X, Y,T)\equiv(\Hat{X},\Hat{Y},\Hat{T})$,
  \item
    $\Hat{X}$ has density character $\omega_1$,
  \item
    $\Hat{Y}$ is separable. 
  \end{itemize}
  Furthermore, if $D$ is a given countable subset of $X$, the
  structure $(\Hat{X},\Hat{Y},\Hat{T})$ can be chosen with the
  following property: there exists a separable subspace $X_0$ of $X$
  such that $D\subseteq X_0$ and if $T_0$ denotes the restriction of $T$ to
  $X_0$,
  \begin{itemize}
  \item
    $(X_0,T_0(X_0),T_0)\prec (X,Y,T)$,
  \item
    $(X_0,T_0(X_0),T_0)\prec (\Hat{X},\Hat{Y},\Hat{T})$.
  \end{itemize}
\end{theorem}
\begin{proof}
By the L\"owenheim-Skolem Theorem~\cite[page 47]{Henson-Iovino:2002}, there exists a separable subspace $X_0$ of $X$ such that  if $T_0$ denotes the restriction of $T$ to $X_0$ and $Y_0=T_0(X_0)$,
\[
(X_0,Y_0,T_0)\prec (X,Y,T).
\]
(Note that $X_0$ can be taken so that it contains any given countable subset of $X$.)

Let $A$ be a countable dense subset of $X_0$ and consider the structure
\[
(\,X_0,Y_0,T,a \mid a\in A\,).
\]
Let $\cL$ be the signature that results from expanding the signature
of $(\,X,Y,T,a \mid a\in A\,)$ with new constant symbols $c_0,c_1,\dots$
and $c_0^*$ of sort $B_X(1)$ as well as new constant symbols
$d_0,d_1,\dots$ and of sort $B_Y(1)$. 

Let us introduce some temporary terminology. If $\phi$ is an
$\cL$-sentence, an $\cL$-structure $M$ \emph{satisfies} the inequality
$\phi<r$ if $\phi^M<r$. An inequality $\phi<r$ will be called \emph{finitary}
if the formula $\phi$ is finitary.

Let $\Gamma$ consist of the following inequalities (the variable $x$ in~\ref{I:Surjectivity} is of sort $B_Y(1)$):

\begin{enumerate}
\item
  \label{item:ElemIneq}
  All the finitary inequalities satisfied by the structure
  $(\,X_0,Y_0,T,a \mid a\in A\,)$
\item
  $\lnot\|c_n\|<\epsilon$ (i.e., $\|c_n\| > 1-\varepsilon$) and $\lnot\|d_n\|<\epsilon$,
  for every $n<\omega$ and every rational $\epsilon>0$.
\item
  $\lnot\half\|c_m-c_n\| < \half + \epsilon$ (i.e., $\|c_n-c_m\| > 1-2\varepsilon$)
  and $\lnot\half\|d_m-d_n\|<\half+\epsilon$,
  for every pair $m,n$ with $m<n<\omega$ and every rational $\epsilon>0$.
\item
  $\lnot \|c_0^*\|<\epsilon$, for every rational $\epsilon>0$.
\item
  $\lnot\half\|c_0^*-c_n\|<\half+\epsilon$, for every $n<\omega$ and every rational $\epsilon>0$.
\item
  $\|T(c_\omega)\|<\epsilon$, for every rational $\epsilon>0$.
\item
  \label{I:Surjectivity}
  For every rational $\epsilon>0$, the inequality
  \[
  \sup_x \left(
    \bigwedge_{\substack{r_0,\dots, r_{n-1}\in\mathbb{Q}\cap[-1,1]\\ n<\omega}}(n+1)\cdot\frac{1}{n+1}\| T(x)-\sum_{i<n} r_id_i\|\right)<\epsilon.
  \]
  Here $\frac{1}{n+1}\|\cdot\|$ is just the normalized norm predicate on the
  sort of $ T(x)-\sum_{i<n} r_id_i $, and
  $(n+1)\cdot\varphi$ is defined in general as $\varphi \dotplus \cdots \dotplus \varphi$ $n+1$
  times.
\end{enumerate}

Lemma~\ref{L:sequence} ensures that the hypotheses of \fref{cor:omitting types}
are satisfied with $\Sigma = \cL_{\omega,\omega}$.
Thus
by \fref{cor:omitting types}, $\Gamma$ has a separable model
\[
(\,X_1,Y_1,T_1, a, a_n, b_n, a_0^*\mid a\in A, n<\omega\,),
\]
where for each $n<\omega$, $a_n$ is the interpretation of $c_n$,  $b_n$ is
the interpretation of $d_n$, and $a_0^*$ is the interpretation of
$c_0^*$.
By \fref{item:ElemIneq}, we have
\[
(\,X, Y,T)\prec
(\,X_1, Y_1,T_1),
\]
 so, in particular, by Proposition~\ref{P:surjectivity}, $T_1$ is surjective.

Now we iterate the preceding process to find for each ordinal $\alpha$ with $0<\alpha<\omega_1$ a separable structure
\[
(\,X_\alpha, Y_\alpha,T_\alpha, a_n, b_n, a_i^* \mid n<\omega, i<\alpha\,)
\]
such that if $0<\alpha<\beta<\omega_1$,
\begin{itemize}
\item
$(\,X_\alpha, Y_\alpha,T_\alpha, a_n, b_n, a_i^*  \mid n<\omega, i<\alpha\,)\prec
(\,X_\beta, Y_\beta,T_\beta, a_n, b_n, a_i^*  \mid n<\omega, i<\alpha\,)$
\item
$ a_i^*\in X_\alpha$ for $i<\alpha$
\item
$\| a_i^*\|=1$ and $\| a_i^*- a_j^*\|=1$ for $i<j<\alpha$
\item
The linear span of $\{b_n\mid n<\omega\}$ is a dense subset of $T_\alpha(X_\alpha)$.
\end{itemize}
The theorem then follows by taking $\Hat{X}=\bigcup_{\alpha<\omega_1} X_\alpha$,
$\Hat{Y}=\bigcup_{\alpha<\omega_1} Y_\alpha$, and 
 $\Hat{T}=\bigcup_{\alpha<\omega_1} T_\alpha$.
\end{proof}

\newcommand{\etalchar}[1]{$^{#1}$}
\providecommand{\bysame}{\leavevmode\hbox to3em{\hrulefill}\thinspace}
\providecommand{\MR}{\relax\ifhmode\unskip\space\fi MR }
\providecommand{\MRhref}[2]{%
  \href{http://www.ams.org/mathscinet-getitem?mr=#1}{#2}
}
\providecommand{\href}[2]{#2}

\end{document}